\documentclass[11pt]{article}
\usepackage{amsmath}
\usepackage{amsfonts}
\usepackage{amssymb}
\usepackage{amsthm}
\usepackage{graphicx}
\usepackage{cleveref}
\usepackage{enumerate}
\usepackage[justification=centering]{caption}
\usepackage{cancel}
\usepackage{epstopdf}
\usepackage{natbib}
\usepackage{fullpage}
\newtheorem{Thm}{Theorem}[section]
\newtheorem{Lem}[Thm]{Lemma}
\newtheorem{Cor}[Thm]{Corollary}

\newcommand{\B}{\mathrm{B}}

\theoremstyle{definition}

\newtheorem{Rem}[Thm]{Remark}

\makeatletter
\def\blfootnote{\xdef\@thefnmark{}\@footnotetext}
\makeatother

\title{Structure of the  third moment \\ of the generalized Rosenblatt distribution}
\author{
        Shuyang Bai \quad Murad Taqqu\\ 
}

\begin{document}
\maketitle
\begin{abstract}
The Rosenblatt distribution appears as limit in non-central limit theorems.
The generalized Rosenblatt distribution  is obtained by allowing different power exponents in the kernel that defines the usual Rosenblatt distribution. We derive an explicit formula for its third moment, correcting the one in \citet{maejima:tudor:2012:selfsimilar} and \citet{tudor:2013:analysis}. Evaluating this formula numerically, we are  able  to confirm that the class of generalized Hermite processes is strictly richer than the class of Hermite processes. 
\end{abstract}

\blfootnote{
\begin{flushleft}
\end{flushleft}
\textbf{2010 AMS Classification:} 60G18, 60F05\\
}
\section{Introduction}
The Rosenblatt process is a non-Gaussian self-similar process with stationary increments. It can be represented by a double Wiener-It\^o integral as follows:
\begin{equation}\label{eq:Rosenblatt proc}
Z_\gamma(t)= A~ \int_{\mathbb{R}^2}' ~\int_0^t  (s-x_1)_+^{\gamma}(s-x_2)_+^{\gamma} ds ~B(dx_1) B(dx_2),
\end{equation}
where $A\neq 0$ is a constant, the prime $'$ indicates the \emph{exclusion} of the diagonals $x_1=x_2$ in the integral, $\gamma\in (-3/4,-1/2)$, and $B(\cdot)$ is a Brownian random measure.  The process is self-similar with  Hurst  index  $H=2\gamma+2\in (1/2,1)$, that is, for any constant $a>0$, $\{Z(at)\}$ and $\{a^HZ(t)\}$ have the same finite-dimensional distributions.

The marginal distribution of $Z_\gamma(t)$, which we call the \emph{Rosenblatt distribution}, was first characterized by \citet{rosenblatt:1961:independence}, and  the Rosenblatt process was then defined in \citet{taqqu:1975:weak}. The Rosenblatt process belongs to a more general class of processes called \emph{Hermite processes}. A $k$-th order Hermite process is defined through a $k$-tuple Wiener It\^o integral with integrand $\int_0^t\prod_{j=1}^k (s-x_j)_+^\gamma ds$ in (\ref{eq:Rosenblatt proc}), where $-1/2-1/(2k)<\gamma<-1/2$. The Rosenblatt process is thus a Hermite process with $k=2$. Hermite processes can appear as limits in  so-called \emph{non-central limit theorems} involving a nonlinear function of  a long-range dependent Gaussian process (\citet{dobrushin:major:1979:non}, \citet{taqqu:1979:convergence}), or  a nonlinear function of a long-range dependent  linear process (\citet{surgailis:1982:zones}, \citet{ho:hsing:1997:limit}).

\citet{maejima:tudor:2012:selfsimilar}  considered the following extension of the Rosenblatt process:
\begin{equation}\label{eq:gen Rosenblatt proc}
Z_{\gamma_1,\gamma_2}(t)=  \frac{A}{2}~ \int_{\mathbb{R}^2}' ~\int_0^t [(s-x_1)_+^{\gamma_1}(s-x_2)_+^{\gamma_2}+(s-x_1)_+^{\gamma_2}(s-x_2)_+^{\gamma_1}] ds ~B(dx_1) B(dx_2),
\end{equation}
where 
\[\gamma_1,\gamma_2\in (-1,-1/2)\text{ and }\gamma_1+\gamma_2>-3/2.\] 
We shall call $Z_{\gamma_1,\gamma_2}(t)$ a \emph{generalized Rosenblatt process}. They computed the second and the third moment of the $Z_{\gamma_1,\gamma_2}(1)$, but unfortunately  their formula for the third moment is incorrect. The third moment will play a crucial role in the identification of the process.

The generalized Rosenblatt process $Z_{\gamma_1,\gamma_2}(t)$ belongs  to a broad class of self-similar process with stationary increments defined on a Wiener chaos called \emph{generalized Hermite process}, which was first introduced by \cite{mori:toshio:1986:law}. See also \citet{bai:taqqu:2013:3-generalized} for details. 

A generalized Hermite process can be represented by a multiple Wiener-It\^o integral  as 
\begin{equation}\label{eq:gen Herm proc}
Z_g(t)=\int_{\mathbb{R}}'~\int_0^t g(s-x_1,\ldots,s-x_k)1_{\{s_1>x_1,\ldots,s_k>x_k\}} ds ~ B(dx_1)\ldots B(dx_k),
\end{equation}
where the nonzero function $g$ is called a \emph{generalized Hermite kernel} (GHK) and is defined by the following two properties:
\begin{enumerate}
\item $g(\lambda x_1,\ldots,\lambda x_k)=\lambda^\alpha g(x_1,\ldots,x_k)$, for some $\alpha\in (-k/2-1/2,-k/2)$;
\item $\int_{\mathbb{R}_+^k} |g(1+x_1,\ldots,1+x_k)g(x_1,\ldots,x_k)| d\mathbf{x}<\infty$.
\end{enumerate}
The first condition is one of homogeneity to ensure that the resulting process $Z_g(t)$ is self-similar. The second condition ensures that the integrand in (\ref{eq:gen Herm proc}) is square integrable. By heuristically interchanging the order of the two integrations $\int_0^t \cdot ds$ and $\int_{\mathbb{R}^k}' \cdot B(dx_1)\ldots B(dx_k)$ in (\ref{eq:gen Herm proc}), the process can be viewed as an integrated process of a stationary nonlinear moving average, which explains the stationary increments of $Z_g(t)$.

Note that  for $Z_{\gamma_1,\gamma_2}$ in (\ref{eq:gen Rosenblatt proc}),  
\[
g(x_1,x_2)=\frac{A}{2}  [x_1^{\gamma_1}x_2^{\gamma_2}+x_1^{\gamma_2}x_2^{\gamma_1}],
\]
and $\alpha=\gamma_1+\gamma_2$.
It follows \citet{bai:taqqu:2013:3-generalized} that $Z_g(t)$ is self-similar with  Hurst index  
\[H=\alpha+k/2+1\in (1/2,1). 
\]
The process $Z_g$ and other related processes appear as limits in various types of non-central limit theorems involving Voterra-type nonlinear process. See  \citet{bai:taqqu:2014:4-convergence} and \citet{bai:taqqu:2014:5-impact} for details. The following is a natural question:
\begin{center}
\emph{Is the class of generalized Hermite processes strictly richer than the class of Hermite processes for a given $k$ and $H$}? \footnote{Processes differing by a multiplicative constant are considered to be the same process.}
\end{center}

Since all generalized Hermite processes are $H$-self-similar with stationary increments, they all have identical covariances up to a multiplicative factor. Hence the covariance cannot be of any help in answering the preceding question. 

In this paper, we answer the preceding question positively by computing explicitly the second and the third moment of the marginal law of the generalized Rosenblatt  process $Z_{\gamma_1,\gamma_2}(t)$ in (\ref{eq:gen Rosenblatt proc}) at $t=1$, namely, the law of $Z_{\gamma_1,\gamma_2}(1)$ which we call the \emph{generalized Rosenblatt distribution}. Since the  second and the third moments can be expressed in terms of beta functions, one can evaluate   the  moments numerically in an  accurate way,  and use them to show that the preceding question has a positive answer.  

\begin{Rem}
The second moment formula (\ref{eq:mu_2}) has been obtained in Lemma 2.2 of \citet{maejima:tudor:2012:selfsimilar}\footnote{\citet{maejima:tudor:2012:selfsimilar} also attempted to compute the third moment, but unfortunately the function $f_{H_1,H_2}(u_1,u_2,u_3)$ in the proof of their Proposition 3.1 was not computed correctly. The exponents in the first and the third factor of $f_{H_1,H_2}(u_1,u_2,u_3)$ should be $H_1-1$ and $H_2-1$ respectively according to their Lemma 2.1. This error was reproduced  in the proof of Proposition 3.10 of \citet{tudor:2013:analysis}.}.
\end{Rem}

The paper is organized as follows. In Section \ref{sec:main}, we state our formulas for the second and the third moments of $Z_{\gamma_1,\gamma_2}(1)$. Section \ref{sec:proof} contains some preliminary lemmas.  Section \ref{sec:proof main} contains the proof of the results of Section \ref{sec:main}. In Section \ref{sec:num}, we present the numerical evaluation of the third moment of a standardized $Z_{\gamma_1,\gamma_2}(1)$ and answer positively the question  stated above.

\section{Main results}\label{sec:main}

The random variable  $Z_{\gamma_1,\gamma_2}(1)$ defined in (\ref{eq:gen Rosenblatt proc}) has mean $\mu_1(\gamma_1,\gamma_2)=0$ since it is expressed as a Wiener-It\^o integral. 
The following theorem provides 
an explicit expression of the second and the third moment of  $Z_{\gamma_1,\gamma_2}(1)$.   
\begin{Thm}\label{Thm:main}
The second moment of $Z_{\gamma_1,\gamma_2}(1)$ is
\begin{align}
\mu_2(\gamma_1,\gamma_2)=&    \frac{A^2}{(\gamma_1+\gamma_2+2)(2(\gamma_1+\gamma_2)+3)}\times\notag\\& 
\big[\B(\gamma_{1}+1,-\gamma_{1}-\gamma_{2}-1) \B(\gamma_{2}+1,-\gamma_{1}-\gamma_{2}-1)+\B(\gamma_{1}+1,-2\gamma_{1}-1) \B(\gamma_{2}+1,-2\gamma_{2}-1)\big],\label{eq:mu_2}
\end{align}
where $\B(x,y)$ denotes the beta function (\ref{eq:beta}).
The third moment of $Z_{\gamma_1,\gamma_2}(1)$ is
\begin{align}
\mu_3(\gamma_1,\gamma_2)=& \frac{2A^3}{(\gamma_1+\gamma_2+2)(3(\gamma_1+\gamma_2)+5)}\times\notag \\&\Big[ \sum_{\sigma\in \{1,2\}^3}
B(\gamma_{\sigma_1}+1,-\gamma_{\sigma_1}-\gamma_{\sigma_3'}-1)
 B(\gamma_{\sigma_1'}+1,-\gamma_{\sigma_1'}-\gamma_{\sigma_2}-1)
B(\gamma_{\sigma_2'}+1,-\gamma_{\sigma_2'}-\gamma_{\sigma_3}-1)\times\nonumber \\&
\qquad \quad ~~B(\gamma_{\sigma_1'}+\gamma_{\sigma_2}+2,\gamma_{\sigma_2'}+\gamma_{\sigma_3}+2)\Big],\label{eq:mu_3}
\end{align}
where  $\sigma=(\sigma_1,\sigma_2,\sigma_3)$ with $\sigma_i=1$ or $2$, and $\sigma'$ is the complement of $\sigma$, namely, $\sigma'_i=3-\sigma_i$.
\end{Thm}

To compare the values of the third moment as $\gamma_1$ and $\gamma_2$ vary, we shall set the variance $\mu_2(\gamma_1,\gamma_2)=1$. By Theorem \ref{Thm:main}, this determines the constant $A$ as:
\begin{equation*}
A(\gamma_1,\gamma_2)=\left(\frac{(\gamma_1+\gamma_2+2)(2(\gamma_1+\gamma_2)+3)}{\B(\gamma_{1}+1,-\gamma_{1}-\gamma_{2}-1) \B(\gamma_{2}+1,-\gamma_{1}-\gamma_{2}-1)+\B(\gamma_{1}+1,-2\gamma_{1}-1) \B(\gamma_{2}+1,-2\gamma_{2}-1)}\right)^{1/2}.
\end{equation*}
Hence
\begin{Cor}\label{Cor}
The third moment of the standardized $Z_{\gamma_1,\gamma_2}(1)$ is 
\begin{align*}
M_3(\gamma_1,\gamma_2)=F_1(\gamma_1,\gamma_2)F_2(\gamma_1,\gamma_2)F_3(\gamma_1,\gamma_2),
\end{align*}
where
\[
F_1(\gamma_1,\gamma_2)=2(\gamma_1+\gamma_2+2)^{1/2}(2(\gamma_1+\gamma_2)+3)^{3/2}(3(\gamma_1+\gamma_2)+5)^{-1},
\]
\begin{align*}
F_2(\gamma_1,\gamma_2)=&\sum_{\sigma\in \{1,2\}^3}
B(\gamma_{\sigma_1}+1,-\gamma_{\sigma_1}-\gamma_{\sigma_3'}-1)
 B(\gamma_{\sigma_1'}+1,-\gamma_{\sigma_1'}-\gamma_{\sigma_2}-1)
B(\gamma_{\sigma_2'}+1,-\gamma_{\sigma_2'}-\gamma_{\sigma_3}-1)\times \\&
\qquad \quad ~~B(\gamma_{\sigma_1'}+\gamma_{\sigma_2}+2,\gamma_{\sigma_2'}+\gamma_{\sigma_3}+2),
\end{align*}
and
\[
F_3(\gamma_1,\gamma_2)=\left[\B(\gamma_{1}+1,-\gamma_{1}-\gamma_{2}-1) \B(\gamma_{2}+1,-\gamma_{1}-\gamma_{2}-1)+\B(\gamma_{1}+1,-2\gamma_{1}-1) \B(\gamma_{2}+1,-2\gamma_{2}-1)\right]^{-3/2}.
\]

\end{Cor}

\section{Preliminary lemmas}\label{sec:proof}
We shall use the following cumulant formula for a double Wiener-It\^o integral (see, e.g., (8.4.3) of \citet{nourdin:peccati:2012:normal}):
\begin{Lem}\label{Lem:double integral cumulant}
If $f$ is a symmetric function in $L^2(\mathbb{R}^2)$, then the $m$-th cumulant of the double Wiener-It\^o integral $X=\int_{\mathbb{R}^2}' f(y_1,y_2)B(dy_1) B(dy_2)$ is given by the following circular integral:
\[
\kappa_m(X)=2^{m-1}(m-1)!  \int_{\mathbb{R}^m} f(y_1,y_2)f(y_2,y_3)\ldots f(y_{m-1},y_m)f(y_m,y_1) dy_1\ldots dy_m.
\]  
\end{Lem} 
Note, however, that for a random variable with zero mean, which is the case for $Z_{\gamma_1,\gamma_2}(1)$, the second and the third cumulants coincide with the second and the third moments respectively.

The following formulas involving the beta function $\B(x,y)$ will be used many times: 
\begin{equation}\label{eq:beta}
\B(x,y):=\int_0^1 u^{x-1}(1-u)^{y-1}du=\int_0^\infty w^{x-1} (1+w)^{-x-y}dw=\frac{\Gamma(x)\Gamma(y)}{\Gamma(x+y)}
\end{equation}
for all $x,y>0$.

\begin{Lem}\label{Lem:useful 1}
For  $a,b\in (-1,-1/2)$,
\[
\int_{\mathbb{R}} (s_1-u)^a_+ (s_2-u)^b_+ du= (s_2-s_1)_+^{a+b+1} B(a+1,-a-b-1) + (s_1-s_2)_+^{a+b+1}\B(b+1,-a-b-1) .
\]
\end{Lem}
\begin{proof}
Suppose without loss of generality $s_1<s_2$, then 
\begin{align*}
\int_{-\infty}^{s_1} (s_1-u)^a (s_2-u)^b du&=(s_2-s_1)^{a+b+1}\int_{-\infty}^{s_1} \left(\frac{s_1-u}{s_2-s_1}\right)^a \left(\frac{s_2-u}{s_2-s_1}\right)^b d\left( \frac{u}{s_2-s_1}\right)\\
&=(s_2-s_1)^{a+b+1} \int_0^\infty w^a(1+w)^b dw, 
\end{align*}
by the change of variable $w=(s_1-u)/(s_2-s_1)$. Note that  $a,b<-1/2$ guarantees that $a+b+1<0$.
\end{proof}

\begin{Lem}\label{Lem:useful 2}
For $a,b>-1$  and $x<y$,  
\[
\int_x^y (u-x)^{a}(y-u)^b  du =(y-x)^{a+b+1} \B(a+1,b+1).
\]
\end{Lem}
\begin{proof}
\begin{align*}
\int_x^y (u-x)^{a}(y-u)^b du& = (y-x)^{a+b+1} \int_x^y \left(\frac{u-x}{y-x}\right)^a \left(\frac{y-u}{y-x}\right)^b d\left(\frac{u}{y-x}\right)
\\ &=(y-x)^{a+b+1} \int_0^1 w^a (1-w)^b dw.
\end{align*}
\end{proof}

\begin{Lem}\label{Lem:key}
For $\beta_j >-1$, $j=1,\ldots,m$, $m\ge 2$, such that $\beta_1+\ldots+\beta_m+m>1$, we have
\begin{align}
&\int_{0<s_1<\ldots <s_m<1} (s_m-s_1)^{\beta_1} (s_2-s_1)^{\beta_2} (s_3-s_2)^{\beta_3}\ldots (s_m-s_{m-1})^{\beta_m} ds_1 \ldots s_m \label{eq:start}
\\=&(m+\beta_1+\ldots+\beta_m)^{-1}(m-1+\beta_1+\ldots+\beta_m)^{-1}\frac{\Gamma(\beta_2+1)\Gamma(\beta_3+1)\ldots \Gamma(\beta_m+1)}{\Gamma(\beta_2+\beta_3+\ldots+\beta_m+m-1)}.\notag
\end{align}
\end{Lem}
\begin{proof}
For convenience set $C_m=(m+\beta_1+\ldots+\beta_m)^{-1}$, and $C_{m-1}'=(m-1+\beta_1+\ldots+\beta_m)^{-1}$.
The starting expression (\ref{eq:start}) can  be written as:
\begin{align*}
&\int_{0<s_1<\ldots <s_m<1} s_m^{\beta_1+\ldots+\beta_m} \left(1-\frac{s_1}{s_m}\right)^{\beta_1} \left(\frac{s_2}{s_m}-\frac{s_1}{s_m}\right)^{\beta_2} \ldots \left(\frac{s_{m-1}}{s_m}-\frac{s_{m-2}}{s_m}\right)^{\beta_{m-1}} \left(1-\frac{s_{m-1}}{s_m}\right)^{\beta_m} ds_1 \ldots s_m\\
=&\int_0^1 s^{\beta_1+\ldots+\beta_m+m-1} ds \int_{0<u_1<\ldots<u_{m-1}<1} (1-u_1)^{\beta_1} \ldots (u_{m-1}-u_{m-2})^{\beta_{m-1}}(1-u_{m-1})^{\beta_m} du_{m-1}\ldots du_{1}\\
=&C_m \int_{0<u_1<\ldots<u_{m-1}<1} (1-u_1)^{\beta_1} (u_2-u_1)^{\beta_2}\ldots (u_{m-1}-u_{m-2})^{\beta_{m-1}}(1-u_{m-1})^{\beta_m} du_{m-1}\ldots du_{1}.
\end{align*}
Integrating over $u_{m-1}$, we get by Lemma \ref{Lem:useful 2} that (\ref{eq:start}) equals
\begin{align*}
&C_m \B(\beta_{m-1}+1,\beta_{m}+1) \int_{0<u_1<\ldots<u_{m-2}<1} (1-u_1)^{\beta_1} \ldots (u_{m-2}-u_{m-3})^{\beta_{m-2}}(1-u_{m-2})^{\beta_{m-1}+\beta_m+1} du_{m-2}\ldots du_{1} .
\end{align*}
Now by repeatedly applying Lemma \ref{Lem:useful 2},  we can write (\ref{eq:start}) as:
\begin{align*}
& C_m  \B(\beta_{m-1}+1,\beta_{m}+1)\B(\beta_{m-2}+1,\beta_{m-1}+\beta_{m}+2)\ldots \B (\beta_2+1,\beta_3+\ldots+\beta_m+m-2)\times \\& \int_0^1 (1-u_1)^{\beta_1} (1-u_1)^{\beta_2+\ldots+\beta_m+m-2}du_1
\\=& C_mC_{m-1}'  \frac{\Gamma(\beta_{m-1}+1)\Gamma(\beta_{m}+1)}{{\Gamma(\beta_{m-1}+\beta_m+2)}} \frac{\Gamma(\beta_{m-2}+1){\Gamma(\beta_{m-1}+\beta_{m}+2)}}{{\Gamma(\beta_{m-2}+\beta_{m-1}+\beta_{m}+2)}}\ldots  \frac{\Gamma(\beta_2+1){\Gamma(\beta_3+\ldots+\beta_m+m-2)}}{\Gamma(\beta_2+\ldots+\beta_m+m-1)}\\
=& (m+\beta_1+\ldots+\beta_m)^{-1}(m-1+\beta_1+\ldots+\beta_m)^{-1}\frac{\Gamma(\beta_2+1)\Gamma(\beta_3+1)\ldots \Gamma(\beta_m+1)}{\Gamma(\beta_2+\beta_3+\ldots+\beta_m+m-1)}.
\end{align*}
\end{proof}

\section{Proof of Theorem \ref{Thm:main}}\label{sec:proof main}
\begin{proof}
Set $g(x,y)=\frac{A}{2}(x_+^{\gamma_1}y_+^{\gamma_2}+x_+^{\gamma_2}y_+^{\gamma_1}),$
and observe that $g$ is symmetric.
In view of Lemma \ref{Lem:double integral cumulant}, we  need to compute the following integral for $m=2$ and $m=3$: 
\begin{equation}\label{eq:c_m}
c_m=\int_{[0,1]^m} d\mathbf{s}I(s_1,\ldots,s_m),
\end{equation}
where  
\begin{equation}\label{eq:I}
I(s_1,\ldots,s_m)=\int_{\mathbb{R}^m} d\mathbf{x} 
g(s_1-x_1,s_1-x_2)g(s_2-x_2,s_2-x_3)\ldots g(s_m-x_m,s_m-x_1).
\end{equation}
The case $m=2$ was done by \citet{maejima:tudor:2012:selfsimilar}. It is instructive, however, to continue using the symbol $m$. 

We claim that for $m=2,3$,  $I(s_1,\ldots,s_m)$ does not change if one permutes $s_1,\ldots,s_m$.  For $m=2$, this is obvious since the integrand is $g(s_1-x_1,s_1-x_2)g(s_2-x_2,s_2-x_1)=g(s_2-x_1,s_2-x_2)g(s_1-x_1,s_1-x_2)$ using the symmetry of $g$. For $m=3$, suppose one switches $s_2$ with $s_3$, then we have by the symmetry of $g$ that
\begin{align*}
&g(s_1-x_1,s_1-x_2)g(s_3-x_2,s_3-x_3)g(s_2-x_3,s_2-x_1)\\
=& g(s_1-x_2,s_1-x_1) g(s_2-x_1,s_2-x_3) g(s_3-x_3,s_3-x_2).
\end{align*}
Now if one changes the sub-indices (which does not affect the integral) of $x_i$'s in the following way: $x_2\rightarrow x_1$, $x_1\rightarrow x_2$, $x_3\rightarrow x_3$,  one gets exactly the original integrand expression:
\[g(s_1-x_1,s_1-x_2)g(s_2-x_2,s_2-x_3)g(s_3-x_3,s_3-x_1). 
\]
Similarly the integral $I(s_1,s_2,s_3)$ does not change  if one switches $s_1$ with $s_3$ or switches $s_2$ with $s_3$. 

Therefore, $I(s_1,\ldots,s_m)$ in (\ref{eq:I}) is a symmetric function for $m=2,3$.\footnote{One can check that the symmetry does not hold for $m\ge 4$, and hence the arguments in this proof only works for $m=2,3$.}
Hence it suffices to focus the integration on 
\[
E_m:= \{(\mathbf{x},\mathbf{s})\in \mathbb{R}^{m}\times [0,1]^m, s_1<\ldots< s_m\}.
\]
Then 
\begin{align*}
c_m=&\int_{[0,1]^m} d\mathbf{s} \int_{\mathbb{R}^m} d\mathbf{x} g(s_1-x_1,s_1-x_2)\ldots g(s_m-x_m,s_m-x_1)\notag\\
=&m!\int_{E_{m}} d\mathbf{s}d\mathbf{x} g(s_1-x_1,s_1-x_2)\ldots g(s_m-x_m,s_m-x_1).
\end{align*}
To evaluate the integral, we  view the indices below modulo $m$, e.g., $x_{m+1}=x_1$ and $s_0=s_m$. Then
\begin{align*}
c_m=& m!A^m2^{-m}\int_{E_{m}} d\mathbf{s}d\mathbf{x}\prod_{i=1}^{m} [(s_{i}-x_{i})_+^{\gamma_1} (s_{i}-x_{i+1})_+^{\gamma_2}+(s_{i}-x_{i})_+^{\gamma_2} (s_{i}-x_{i+1})_+^{\gamma_1}]\\
=&m!A^m2^{-m}\sum_{\sigma\in \{1,2\}^m} \int_{E_{m}} d\mathbf{s}d\mathbf{x} \prod_{i=1}^{m} (s_{i}-x_i)_+^{\gamma_{\sigma_i}}(s_{i}-x_{i+1})_+^{\gamma_{\sigma_{i}'}},
\end{align*}
where if $\sigma_i=1$ then $\sigma_i'=2$ and vice versa.

Now since $(s_1-x_1)_+^{\gamma_{\sigma_1}}(s_0-x_1)_+^{\gamma_{\sigma_m'}}=(s_1-x_1)_+^{\gamma_{\sigma_1}}(s_m-x_1)_+^{\gamma_{\sigma_m'}}$, we can reorder the terms in the product and write using Lemma \ref{Lem:useful 1},
\begin{align}
c_m=&m!A^m2^{-m}\sum_{\sigma\in \{1,2\}^m} \int_{E_{m}} d\mathbf{s}d\mathbf{x} \prod_{i=1}^{m} (s_{i}-x_i)_+^{\gamma_{\sigma_i}}(s_{i-1}-x_i)_+^{\gamma_{\sigma_{i-1}'}} \notag
\\
=&m!A^m2^{-m}\sum_{\sigma\in \{1,2\}^m} \int_{0<s_1<\ldots<s_m<1} d\mathbf{s}  、\int_{\mathbb{R}}(s_1-x_1)_+^{\gamma_{\sigma_1}}(s_{m}-x_1)_+^{\gamma_{\sigma_{m}'}}  dx_1 \prod_{i=2}^{m} \int_\mathbb{R} (s_{i-1}-x_i)_+^{\gamma_{\sigma_{i-1}'}} (s_{i}-x_i)_+^{\gamma_{\sigma_i}} dx_i \notag
\\
=&m!A^m2^{-m} \sum_{\sigma\in \{1,2\}^m}\left[ \B(\gamma_{\sigma_1}+1,-\gamma_{\sigma_m'}-\gamma_{\sigma_1}-1)\prod_{i=2}^m \B(\gamma_{\sigma_{i-1}'}+1,-\gamma_{\sigma_{i-1}'}-\gamma_{\sigma_{i}}-1)\right]J_\sigma,\label{eq:J_sigma}
\end{align}
where 
\[
J_{\sigma}=\int_{0<s_1<\ldots<s_m<1} (s_m-s_1)^{\gamma_{\sigma_m'}+\gamma_{\sigma_1}+1} \prod_{i=2}^m  (s_{i}-s_{i-1})^{\gamma_{\sigma_{i-1}'}+\gamma_{\sigma_i}+1} d\mathbf{s}.
\]
Applying Lemma \ref{Lem:key} to $J_\sigma$, by setting  $\beta_1=\gamma_{\sigma_m'}+\gamma_{\sigma_1}+1$, $\beta_i=\gamma_{\sigma_{i-1}'}+\gamma_{\sigma_i}+1$ for $i=2,\ldots,m$, one gets 
\begin{align*}
J_{\sigma}=
&(m+\beta_1+\ldots+\beta_m)^{-1}(m-1+\beta_1+\ldots+\beta_m)^{-1}\frac{\Gamma(\beta_2+1)\Gamma(\beta_3+1)\ldots \Gamma(\beta_m+1)}{\Gamma(\beta_2+\beta_3+\ldots+\beta_m+m-1)}.
\end{align*}
Since $\gamma_{\sigma_j}+\gamma_{\sigma_j'}=\gamma_1+\gamma_2$, we have 
\[
\sum_{i=1}^m \beta_i=\gamma_{\sigma_m'}+\gamma_{\sigma_1}+\ldots+\gamma_{\sigma_1'}+\gamma_{\sigma_m}+m=m(\gamma_1+\gamma_2+1)
\]
and
\[
\sum_{i=2}^m \beta_i=\gamma_{\sigma_1'}+(\gamma_{\sigma_2}+\gamma_{\sigma_2'})+\ldots+(\gamma_{\sigma_{m-1}}+\gamma_{\sigma_{m-1}'})+\gamma_{\sigma_m}+(m-1),
\]
where $\sum_{i=1}^m \beta_i+m=m(\gamma_1+\gamma_2+2)>1$ because $\gamma_1+\gamma_2>-3/2$ and $m\ge 2$, and hence Lemma \ref{Lem:key} applies.
This yields
\begin{align*}
J_{\sigma}=
m^{-1}[\gamma_1+\gamma_2+2]^{-1}[m(\gamma_1+\gamma_2)+2m-1]^{-1}\frac{\prod_{i=2}^m\Gamma(\gamma_{\sigma_{i-1}'}+\gamma_{\sigma_{i}}+2)}{\Gamma\big(\gamma_{\sigma_{1}'}+\gamma_{\sigma_{m}}+(m-2)(\gamma_1+\gamma_2)+2(m-1)\big)}.
\end{align*}
Plugging this $J_\sigma$ in the expression of $c_m$ in  (\ref{eq:J_sigma}) and  using Lemma \ref{Lem:double integral cumulant}, we have
\begin{equation}\label{eq:mu_m}
\mu_m(\gamma_1,\gamma_2)=2^{m-1} (m-1)!c_m. 
\end{equation}

Suppose first $m=2$. In this case, summing over $\sigma\in \{1,2\}^2$ in (\ref{eq:J_sigma}) means letting $\sigma$ take the values $(1,1)$, $(1,2)$, 
$(2,1)$ and $(2,2)$. We then gain a factor of $2$, because, by symmetry, the terms in (\ref{eq:J_sigma}) corresponding to $(1,1)$ and $(2,2)$ are identical and so are the terms corresponding to $(1,2)$ and $(2,1)$. Thus (\ref{eq:mu_m}) yields (\ref{eq:mu_2}).

In the case $m=3$, we have
\[
J_{\sigma}=
3^{-1}[\gamma_1+\gamma_2+2]^{-1}[3(\gamma_1+\gamma_2)+5]^{-1}\frac{\Gamma(\gamma_{\sigma_{1}'}+\gamma_{\sigma_{2}}+2)\Gamma(\gamma_{\sigma_{2}'}+\gamma_{\sigma_{3}}+2)}{\Gamma\big(\gamma_{\sigma_{1}'}+\gamma_{\sigma_{m}}+\gamma_1+\gamma_2+4\big)}.
\]
So (\ref{eq:mu_m}) yields (\ref{eq:mu_3}) using the last equality in (\ref{eq:beta}). This completes the proof of Theorem \ref{Thm:main}.

\end{proof}

\section{Numerical evaluation of the third moment}\label{sec:num}

We shall show  that the class of  generalized Hermite distributions strictly contains the class of Hermite distributions. More specifically, we show that the class of generalized Rosenblatt distribution strictly contains the class of Rosenblatt distributions. For this purpose, we  restrict throughout the variance 
\[\mu_2(\gamma_1,\gamma_2)=1, 
\]
and compute numerically the third moment $M_3(\gamma_1,\gamma_2)$ as given in Corollary \ref{Cor}. Figure \ref{fig:contour} displays a contour plot of the third moment $\mu_3(\gamma_1,\gamma_2)$ in (\ref{eq:mu_3}). 

\begin{figure}
\centering
\includegraphics[scale=0.9]{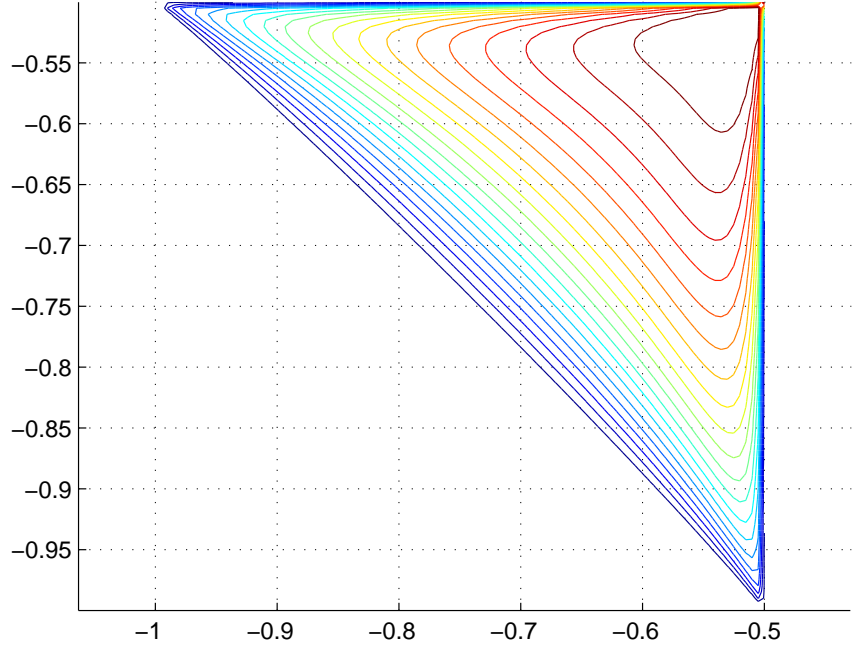}
\caption{Contour plot of $\mu_3(\gamma_1,\gamma_2)$. \\{\small
Boundaries are given by the lines $\gamma_1=-1/2$, $\gamma_2=-1/2$ and $\gamma_1+\gamma_2=-3/2$.}}\label{fig:contour}
\end{figure}

We shall also fix  $\alpha=\gamma_1+\gamma_2$, or equivalently, fix the Hurst index $H=\alpha+2$, and show that the third moment $M_3(\gamma_1,\gamma_2)$ does change when $\gamma_1$  changes and $\gamma_2=\alpha-\gamma_1$.

In Tables 1-4 and Figures 1-4, we list and plot the values of 
\[
\text{$M_3(\gamma_1,\alpha-\gamma_1)$ against $\gamma_1$ for $H=0.6,0.7,0.8,0.9$.}
\] 
\begin{Rem}
Due to the symmetry, $M_3(\gamma_1,\gamma_2)=M_3(\gamma_2,\gamma_1)$. Recall that $\gamma_1,\gamma_2\in (-1,-1/2)$ with $\gamma_1+\gamma_2>-3/2$. Thus $\alpha=\gamma_1+\gamma_2\in (-3/2,-1)$ and $H=\alpha+2\in (1/2,1)$. In Tables 1-4 we let $\gamma_1$ take values from $\alpha/2$ to $-0.505$. 
\end{Rem}

\begin{Rem}
If $\gamma_1=\gamma_2$, then $\gamma_1=\gamma_2=\alpha/2$, and  
$M_3(\alpha/2,\alpha/2)$ becomes the third moment of the  standardized Rosenblatt distribution $Z_{\alpha/2}(1)$ (see (\ref{eq:Rosenblatt proc})). Its values (given in the first column in the tables)  coincide with those obtained in \citet{veillette:taqqu:2013:properties}. See Table 4 of the supplement of \citet{veillette:taqqu:2013:properties}, where they are listed as a function of  the parameter $D=1-H$.
\end{Rem}

Since $M_3(\gamma_1,\gamma_2)$ varies with $\gamma_1+\gamma_2=\alpha$ fixed, we conclude that  the class of  generalized Hermite distributions is strictly richer than the class of Hermite distributions.

%\clearpage

\begin{table}
\centering
\begin{tabular}{|c|c|c|c|c|c|c|c|c|c|c|}
\hline
$\gamma_1$ &-0.700 & -0.678 & -0.657 & -0.635 & -0.613 & -0.592 & -0.570 & -0.548 & -0.527 & -0.505 \\
\hline
$M_3(\gamma_1,\alpha-\gamma_1)$ &1.183 & 1.189 & 1.206 & 1.236 & 1.281 & 1.340 & 1.413 & 1.486 & 1.488 & 0.947 \\
\hline
\end{tabular}
\caption{$M_3(\gamma_1,\alpha-\gamma_1)$ when $\alpha=-1.4$  (or $H=0.6$).} 
\end{table}
\begin{figure}
\centering
\includegraphics[scale=0.6]{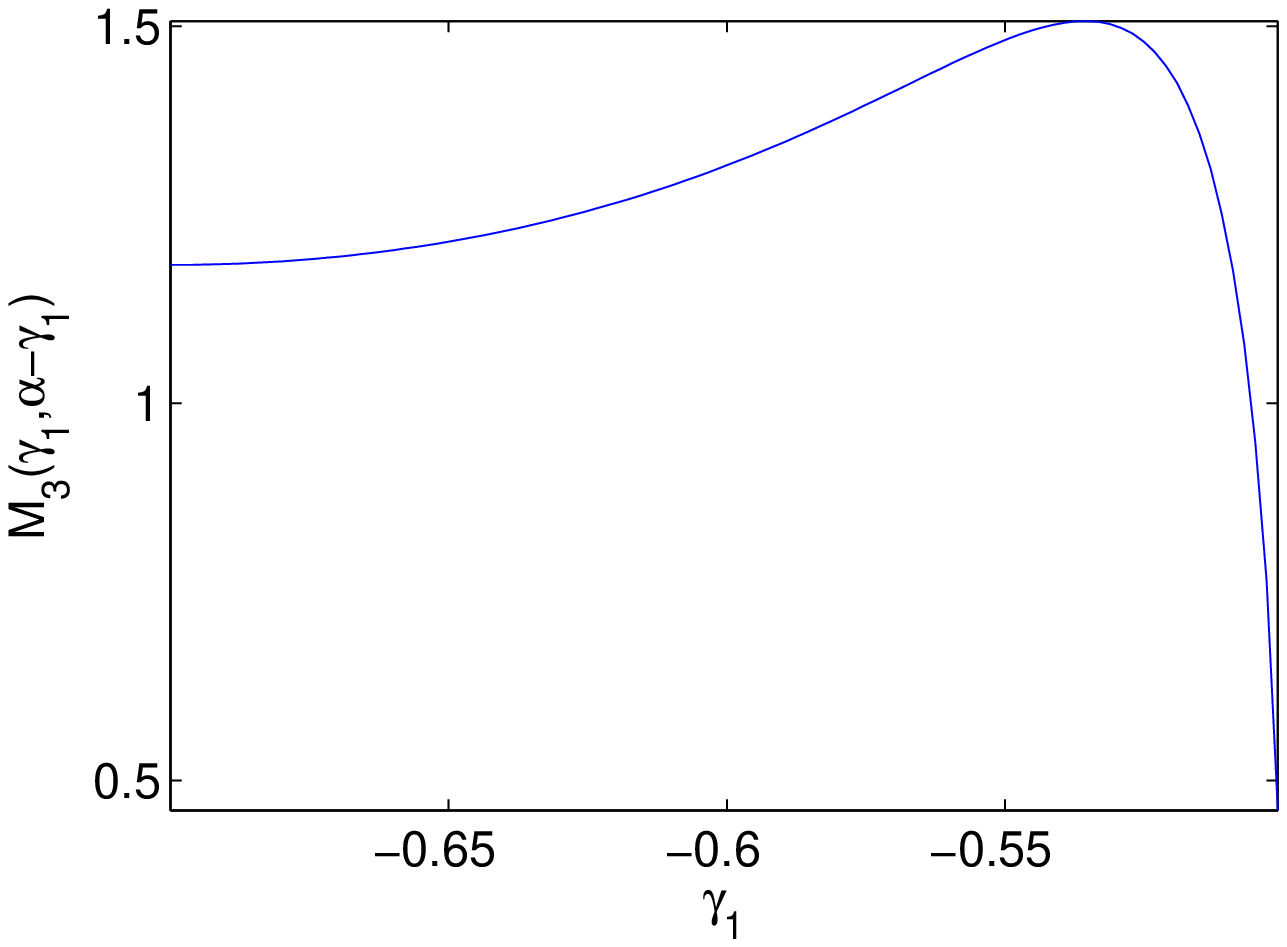}
\caption{$M_3(\gamma_1,\alpha-\gamma_1)$ when $\alpha=-1.4$  (or $H=0.6$).} 
\end{figure}

\begin{table}
\centering
\begin{tabular}{|c|c|c|c|c|c|c|c|c|c|c|}
\hline
$\gamma_1$ &-0.650 & -0.634 & -0.618 & -0.602 & -0.586 & -0.569 & -0.553 & -0.537 & -0.521 & -0.505 \\
\hline
$M_3(\gamma_1,\alpha-\gamma_1)$ &2.067 & 2.071 & 2.082 & 2.101 & 2.125 & 2.149 & 2.162 & 2.135 & 1.972 & 1.239 \\
\hline
\end{tabular}
\caption{$M_3(\gamma_1,\alpha-\gamma_1)$ when $\alpha=-1.3$  (or $H=0.7$).} 
\end{table}
\begin{figure}
\centering
\includegraphics[scale=0.6]{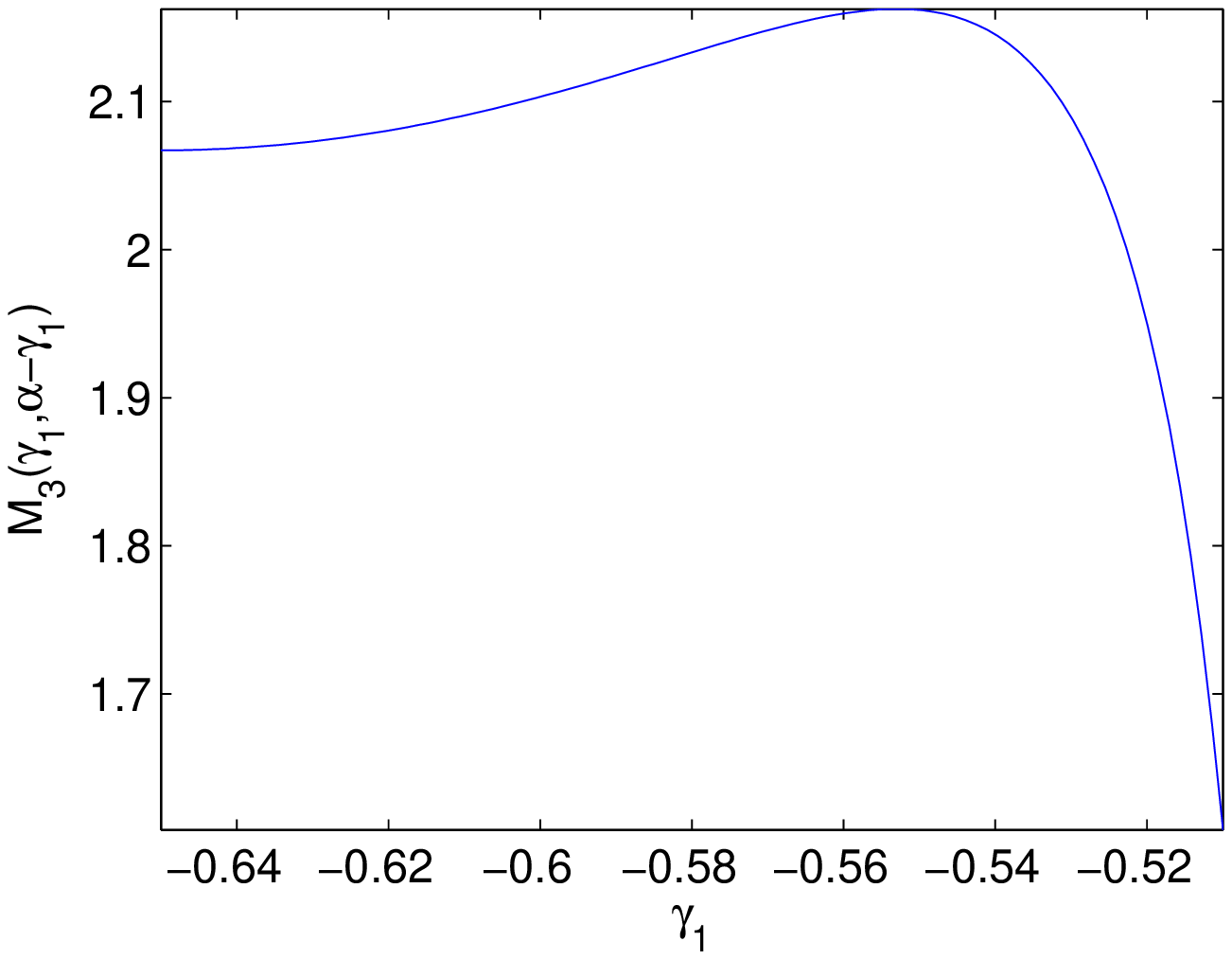}
\caption{$M_3(\gamma_1,\alpha-\gamma_1)$ when $\alpha=-1.3$  (or $H=0.7$).} 
\end{figure}

\begin{table}
\centering
\begin{tabular}{|c|c|c|c|c|c|c|c|c|c|c|}
\hline
$\gamma_1$ &-0.600 & -0.589 & -0.579 & -0.568 & -0.558 & -0.547 & -0.537 & -0.526 & -0.516 & -0.505 \\
\hline
$M_3(\gamma_1,\alpha-\gamma_1)$ &2.548 & 2.549 & 2.554 & 2.559 & 2.564 & 2.561 & 2.538 & 2.465 & 2.258 & 1.587 \\
\hline
\end{tabular}
\caption{$M_3(\gamma_1,\alpha-\gamma_1)$ when $\alpha=-1.2$  (or $H=0.8$).} 
\end{table}
\begin{figure}
\centering
\includegraphics[scale=0.6]{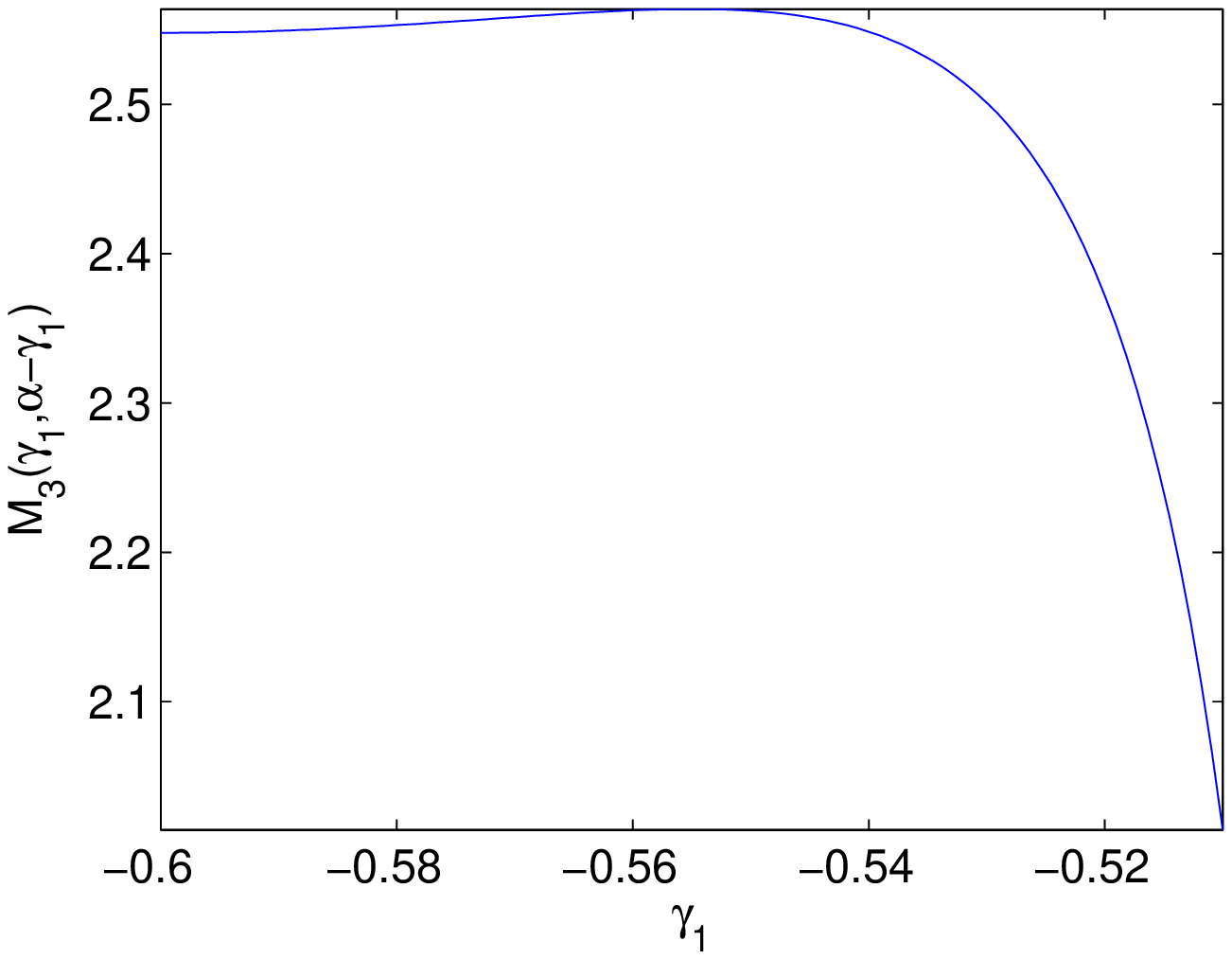}
\caption{$M_3(\gamma_1,\alpha-\gamma_1)$ when $\alpha=-1.2$  (or $H=0.8$).} 
\end{figure}

\begin{table}
\centering
\begin{tabular}{|c|c|c|c|c|c|c|c|c|c|c|}
\hline
$\gamma_1$ &-0.550 & -0.545 & -0.540 & -0.535 & -0.530 & -0.525 & -0.520 & -0.515 & -0.510 & -0.505 \\
\hline
$M_3(\gamma_1,\alpha-\gamma_1)$& 2.770 & 2.770 & 2.770 & 2.770 & 2.766 & 2.755 & 2.726 & 2.659 & 2.505 & 2.113 \\
\hline
\end{tabular}
\caption{$M_3(\gamma_1,\alpha-\gamma_1)$ when $\alpha=-1.1$  (or $H=0.9$).} 
\end{table}

\begin{figure}
\centering
\includegraphics[scale=0.6]{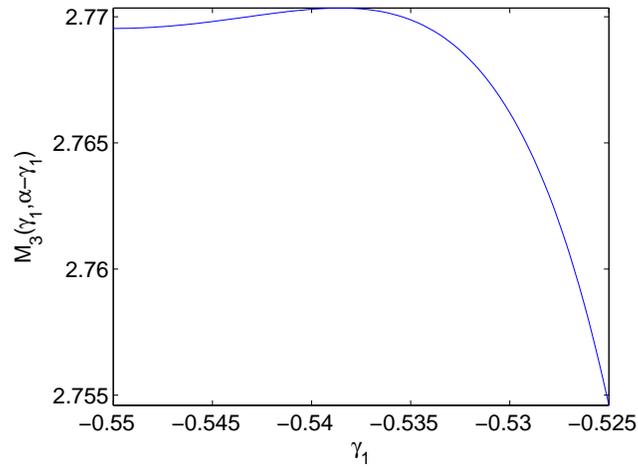}
\caption{$M_3(\gamma_1,\alpha-\gamma_1)$ when $\alpha=-1.1$  (or $H=0.9$).} 
\end{figure}

%\clearpage

%\begin{figure}
%\includegraphics[scale=0.7]{}
%\centering
%\caption{$M_3(\gamma)/48$}\label{fig:K3}
%\end{figure}

\medskip

\noindent\textbf{Acknowledgments.} We would like to thank the referee for noting an error in the original version, and for some other  comments leading to the improvement of the paper. This work was partially supported by the NSF grants DMS-1007616 and DMS-1309009 at Boston University.

\medskip
\noindent Shuyang Bai~~~~~~~ \textit{bsy9142@bu.edu}\\
Murad S. Taqqu ~~\textit{murad@bu.edu}\\
Department of Mathematics and Statistics\\
111 Cumminton Street\\
Boston, MA, 02215, US
\end{document}